\documentclass[fontsize=11pt,a4paper,reqno,numbers=noendperiod,bibliography=totoc]{scrartcl}
\usepackage[USenglish]{babel}
\usepackage[T1]{fontenc}
\usepackage[latin10]{inputenc} 
\usepackage{fouriernc}
\usepackage{amsmath,amssymb,amsthm}
\usepackage{dsfont,bm}
\usepackage[obeyspaces]{url}
\usepackage{colortbl,color,xcolor}
\usepackage{graphicx,tikz,pgfplots}
\usepackage{centernot}
\usepackage{enumitem}
\usepackage[babel]{microtype}
\usepackage{subfigure}
\usepackage{scrhack}
\usepackage{float}
\usepackage{cite}
\usepackage{algorithm}
\usepackage{myalgpseudocode}

\usepackage[bookmarks=true,plainpages=false,linktocpage,colorlinks=true,citecolor=green!80!black,linkcolor=red!70!black,filecolor=magenta,urlcolor=magenta,breaklinks,pdfauthor={Thomas Jahn},pdftitle={Geometric Algorithms For Minimal Enclosing Disks In Strictly Convex Normed Planes}]{hyperref}

\usetikzlibrary{calc,intersections,through,backgrounds,arrows,patterns,shapes.geometric,shapes.misc}

\newcommand{\RR}{\mathds{R}}
\newcommand{\fall}{\:\forall\:}

\newcommand{\mnorm}[1]{\left\lVert#1\right\rVert}
\newcommand{\mnorms}[1]{\lVert#1\rVert}
\newcommand{\setn}[1]{\left\{#1\right\}}
\newcommand{\setns}[1]{\{#1\}}
\newcommand{\setcond}[2]{\left\{#1 \:\middle\vert\: #2\right\}}
\newcommand{\defeq}{\mathrel{\mathop:}=}
\newcommand{\dah}{i.e., }
\newcommand{\lr}[1]{\!\left(#1\right)}
\newcommand{\mc}[2]{B(#1,#2)}
\newcommand{\ms}[2]{S(#1,#2)}
\newcommand{\strline}[2]{\left\langle#1,#2\right\rangle}
\newcommand{\enquote}[1]{``#1''}
\newcommand{\MEBC}{C}
\newcommand{\MER}{R}

\theoremstyle{plain}
\newtheorem{Satz}{Theorem}[section]
\newtheorem{Lem}[Satz]{Lemma}
\newtheorem{Prop}[Satz]{Proposition}

\theoremstyle{definition}
\newtheorem{Def}[Satz]{Definition}
\newtheorem{Bem}[Satz]{Remark}
\newtheorem{Alg}[Satz]{Algorithm}

\DeclareMathOperator{\co}{conv}

\DeclareMathOperator{\bd}{bd}
\DeclareMathOperator{\diam}{diam}
\DeclareMathOperator{\bisec}{bis}
\DeclareMathOperator*{\card}{card}
\let \eps \varepsilon

\addtokomafont{caption}{\small}
\setkomafont{captionlabel}{\bfseries}

\begin{document}
\parindent 0pt
\title{Geometric Algorithms for Minimal Enclosing Disks in Strictly Convex Normed Planes}
\author{Thomas Jahn\\
{\small Faculty of Mathematics, Technische Universit\"at Chemnitz}\\
{\small 09107 Chemnitz, Germany}\\
{\small thomas.jahn\raisebox{-1.5pt}{@}mathematik.tu-chemnitz.de}
}
\date{}
\maketitle

\begin{abstract}
With the geometric background provided by Alonso, Martini, and Spirova on the location of circumcenters of triangles in normed planes, we show the validity of the Elzinga--Hearn algorithm and the Shamos--Hoey algorithm for solving the minimal enclosing disk problem in strictly convex normed planes.
\end{abstract}



\section{Introduction}
In 1857, Sylvester \cite{Sylvester1857} posed the \emph{minimal enclosing disk problem}, which asks for the smallest disk which covers a given finite point set in the Euclidean plane. This problem has been tackled by several authors since the development of computational geometry in the 1970s. For instance, the celebrated algorithm by Megiddo \cite{Megiddo1983b} provides a linear-time solution for the minimal enclosing disk problem. Welzl \cite{Welzl1991} proposes a a randomized algorithm which works well for arbitrary finite dimensions. A natural extension of Sylvester's problem can be obtained by replacing the family of Euclidean disks by the family $\mathcal{F}$ of homothetic images of a given non-empty, convex, compact set $B\subset\RR^2$ which is centrally symmetric with respect to its interior point $\bm{0}=(0,0)$ (the \emph{origin}). The family $\mathcal{F}$ becomes the family of disks with respect to a suitable norm $\mnorm{\cdot}$ on $\RR^2$. Clearly, $B$ and $\mnorm{\cdot}$ are combined by the two relations $B=\setcond{\bm{x}}{\mnorm{\bm{x}}\leq 1}$ and, for $\bm{x}\in\RR^2$, $\mnorm{\bm{x}}=\inf\setcond{\lambda>0}{\bm{x}\in\lambda B}$.
We write $\mc{\bm{x}}{\lambda}=\lambda B+\bm{x}$ and $\ms{\bm{x}}{\lambda}=\lambda \bd(B)+\bm{x}$ for the \emph{disk} (\dah a ball in two dimensions) and the \emph{circle} centered at $\bm{x}$ with radius $\lambda$, respectively. The pair $(\RR^2,\mnorm{\cdot})$ is called a \emph{normed plane}. For a compact set $P\subset\RR^2$, the \emph{minimal enclosing disk problem} is posed by
\begin{equation}
\inf_{\bm{x}\in\RR^2}\max_{\bm{p}\in P}\mnorm{\bm{x}-\bm{p}}.\label{eq:meb_problem}
\end{equation}
The existence of solutions can be shown by standard compactness arguments. We denote the solution set of \eqref{eq:meb_problem} by $\MEBC(P)$, \dah $\MEBC(P)$ is the set of centers of disks that have smallest possible radius and contain $P$. The optimal value of \eqref{eq:meb_problem}, \dah the corresponding radius, is denoted by $\MER(P)$. In general, the minimal enclosing disk problem is not uniquely solvable, as depicted in Figure~\ref{fig:non-uniqueness}.

\begin{figure}[h!]
\begin{center}
\begin{tikzpicture}[line cap=round,line join=round,>=stealth,x=1cm,y=1cm]
\draw (-1,-1)--(1,-1)--(1,1)--(-1,1)--cycle;
\draw[shift={(0.5,0)}] (-1,-1)--(1,-1)--(1,1)--(-1,1)--cycle;
\fill(-0.5,-1) circle (2pt);
\fill (1,-1) circle (2pt);
\fill (0.25,1) circle (2pt);
\end{tikzpicture}
\end{center}
\caption{Minimal enclosing disks need not to be unique. The dots mark the vertices of a triangle $P$, and the figure shows two minimal enclosing disks with respect to the $\ell_\infty$-norm.}\label{fig:non-uniqueness}
\end{figure}
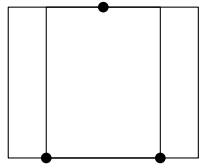

For the Euclidean norm, two algorithmic approaches are known which rely on simple geometric concepts; that of Elzinga and Hearn \cite{ElzingaHe1972a} and that of Shamos and Hoey \cite{ShamosHo1975}. These geometric concepts are given by the notions of obtuseness, rightness, and acuteness of triangles as well as the notion of Voronoi diagrams. The purpose of this article is to show how these concepts can be generalized to a wider class of norms (see Section~\ref{chap:theory}) and how the corresponding algorithms can be proved to be valid (see Section~\ref{chap:algorithms}).

\section{Strictly convex norms, triangles, and Voronoi diagrams}\label{chap:theory}
A norm $\mnorm{\cdot}$ on $\RR^2$ is called \emph{strictly convex} if $\mnorm{\bm{x}+\bm{y}}< 2$ whenever $\mnorm{\bm{x}}=\mnorm{\bm{y}}=1$. Geometrically this means that the boundary of the unit disk does not contain any non-degenerate line segments. Throughout this article, we shall work in  \emph{strictly convex normed planes} $(\RR^2,\mnorm{\cdot})$. At first, let us fix the notation for some geometric entities. The \emph{straight line} through $\bm{x}$ and $\bm{y}$ is denoted by
\begin{equation*}
\strline{\bm{x}}{\bm{y}}=\setcond{\lambda \bm{x}+(1-\lambda)\bm{y}}{\lambda\in\RR}.
\end{equation*}
By the term \emph{triangle}, we understand a set $P\subset\RR^2$ of cardinality $\card(P)=3$. If $\bm{x}$ satifies the equation
\begin{equation*}
\mnorm{\bm{x}-\bm{p}_1}=\mnorm{\bm{x}-\bm{p}_2}=\mnorm{\bm{x}-\bm{p}_3},
\end{equation*}
then $\bm{x}$ is called a \emph{circumcenter} of the triangle $\setn{\bm{p}_1,\bm{p}_2,\bm{p}_3}$. In that case, we call the disk $\mc{\bm{x}}{\mnorm{\bm{x}-\bm{p}_1}}$ a \emph{circumdisk} of $P$. Note that in strictly convex normed planes, the number of circumcenters of each triangle is either $0$ or $1$ \cite[Proposition~14~8.]{MartiniSwWe2001}. 
\begin{Prop}[{\cite[Lemma~2.1.1.1]{Ma2000}}]\label{prop:topological_hyperplanes}
Let $\bm{p}_1,\bm{p}_2\in\RR^2$ be distinct points in a strictly convex normed plane. The bisector 
\begin{equation*}
\bisec(\bm{p}_1,\bm{p}_2)=\setcond{\bm{x}\in \RR^2}{\mnorm{\bm{x}-\bm{p}_1}=\mnorm{\bm{x}-\bm{p}_2}}
\end{equation*}
is homeomorphic to a straight line.
\end{Prop}
\begin{Def}[{see \cite[p.~159]{ShamosHo1975}}]
Let $P\subset\RR^2$ be a given finite point set. The \emph{farthest-point Voronoi region of $\bm{p}\in P$} is defined as
\begin{equation*}
\setcond{\bm{y}\in\RR^2}{\mnorm{\bm{y}-\bm{p}}\geq \mnorm{\bm{y}-\bm{q}}\fall \bm{q}\in P\setminus \setn{\bm{p}}}.
\end{equation*}
The collection of all farthest-point Voronoi regions is said to be the \emph{farthest-point Voronoi diagram}.
\end{Def}
By Proposition~\ref{prop:topological_hyperplanes}, bisectors do not have interior points in strictly convex normed planes. Therefore, the boundary of a farthest-point Voronoi region consists of pieces of curves without endpoints (loci of points belonging to exactly two farthest-point Voronoi regions) and their endpoints (points belonging to at least three farthest-point Voronoi regions). The former are called \emph{edges}, and the latter are called \emph{vertices} of the diagram. For given $\bm{x}\in\RR^2$ and an arbitrary finite set $P\subset\RR^2$ with convex hull $\co P$, one can easily verify the equality $\sup\setcond{\mnorm{\bm{y}-\bm{x}}}{\bm{y}\in \co P}=\sup\setcond{\mnorm{\bm{p}-\bm{x}}}{\bm{p}\in P}$. Due to this fact and the strict convexity of the norm, the points of $P$, which are farthest from $x$, are necessarily extreme points of $\co P$. In particular, the farthest-point Voronoi region of $\bm{p}\in P$ is empty, if $\bm{p}$ is not an extreme point of $\co P$.\\[\baselineskip]
The next two lemmas describe how one half of the bisector can be parametrized by the distance from the two points generating the bisector.
\begin{Lem}\label{lem:continuous_bijection}
Let $\varphi:[0,\infty)\to[\mu_0,\infty)$ be a continuous bijection. Then $\varphi$ is strictly increasing.
\end{Lem}

\begin{Lem}[{\cite[Lemma~2.2]{Vaeisaelae2013}}]\label{lem:parametrization}
Let $\bm{p}_1, \bm{p}_2\in\RR^2$. Furthermore, let $H^+$ be one of the closed half planes bounded by the straight line $\strline{\bm{p}_1}{\bm{p}_2}$, define $\bisec^+(\bm{p}_1,\bm{p}_2)=\bisec(\bm{p}_1,\bm{p}_2)\cap H^+$, and let $\gamma:[0,\infty)\to \bisec^+(\bm{p}_1,\bm{p}_2)$ be a homeomorphism. Then the mapping $\varphi:[0,\infty)\to[\frac{1}{2}\mnorm{\bm{p}_1-\bm{p}_2},\infty)$, $\varphi(t)=\mnorm{\bm{p}_1-\gamma(t)}$, is strictly increasing.
\end{Lem}

\begin{proof}
Obviously, $\varphi$ is continuous. Let us assume that $\varphi$ is not injective. Then there exist two distinct points $\bm{x},\bm{y}\in \bisec^+(\bm{p}_1,\bm{p}_2)$ such that $\lambda\defeq \mnorm{\bm{p}_1-\bm{x}}=\mnorm{\bm{p}_1-\bm{y}}$. By \cite[Corollary~3.1(a)]{AlonsoMaSp2012a}, $\ms{\bm{p}_1}{\lambda}\cap \ms{\bm{p}_2}{\lambda}$ contains the whole segment $[\bm{x},\bm{y}]$. This contradicts the strict convexity. The mapping $\varphi$ is also surjective. Indeed, for $\mu\in[\frac{1}{2}\mnorm{\bm{p}_1-\bm{p}_2},\infty)$, the intersection $\ms{\bm{p}_1}{\mu}\cap\ms{\bm{p}_2}{\mu}\cap H^+$ is a singleton (see \cite[Proposition~14~3.]{MartiniSwWe2001}) which, by definition, belongs to $\bisec^+(\bm{p}_1,\bm{p}_2)$. We have $\gamma(0)=\frac{1}{2}(\bm{p}_1+\bm{p}_2)$, and thus $\varphi(0)=\frac{1}{2}\mnorm{\bm{p}_1-\bm{p}_2}$. By Lemma~\ref{lem:continuous_bijection}, the assertion follows.
\end{proof}
For the sake of completeness, we cite two propositions each of which is crucial both for Theorem~\ref{thm:rademacher} and the understanding of the algorithms in Section~\ref{chap:algorithms}.
\begin{Prop}[{\cite[Lemma~1.2]{AmirZi1980}}]\label{lem:uniqueness}
Let $(\RR^2,\mnorm{\cdot})$ be a normed plane. The norm is strictly convex if and only if $\MEBC(P)$ is a singleton for every compact set $P\subset\RR^2$.
\end{Prop}

\begin{Prop}[{\cite[Section~(1.7)]{GritzmannKl1992}}]\label{lem:gritzmann}
Let $(\RR^2,\mnorm{\cdot})$ be a strictly convex normed plane. Let $(\bm{c}, \lambda)\in \RR^2\times [0,+\infty)$ and $\bm{x}, \bm{y}\in \mc{\bm{c}}{\lambda}$. If $\mnorm{\bm{x}-\bm{y}}=\diam(\mc{\bm{c}}{\lambda})=2\lambda$, then $\frac{1}{2}(\bm{x}+\bm{y})=\bm{c}$.  
\end{Prop}

Rademacher and Toeplitz \cite[Chapter~16]{RademacherToe1994} proved the following theorem for the Euclidean plane. We give an extension for strictly convex normed planes.

\begin{Satz}\label{thm:rademacher}
Let $n\geq 2$, and let $P=\setn{\bm{p}_1,\ldots,\bm{p}_n}$ be a finite set in the strictly convex normed plane $(\RR^2, \mnorm{\cdot})$. Further, let $\mc{\bar{\bm{x}}}{\bar{\lambda}}$ be the minimal enclosing disc of $P$. Then $\card(\ms{\bar{\bm{x}}}{\bar{\lambda}}\cap P)\geq 2$, and every semicircle of $\ms{\bar{\bm{x}}}{\bar{\lambda}}$ (that is, the intersection of the circle with a closed half plane) contains at least one point from $P$.
\end{Satz}
\begin{proof}
Suppose $\card(\ms{\bar{\bm{x}}}{\bar{\lambda}}\cap P)=0$. Then $\mnorm{\bar{\bm{x}}-\bm{p}_i}<\bar{\lambda}$ for all $i\in\setn{1,\ldots,n}$. Hence, the disk centered at $\bar{\bm{x}}$ with radius $\max_{i=1,\ldots,n}\mnorm{\bar{\bm{x}}-\bm{p}_i}$ contains $P$ but has smaller radius. This contradicts $\bar{\lambda}=\MER(P)$.

Suppose $\card(\ms{\bar{\bm{x}}}{\bar{\lambda}}\cap P)=1$, say $\mnorm{\bar{\bm{x}}-\bm{p}_1}=\bar{\lambda}>\mnorm{\bar{\bm{x}}-\bm{p}_i}$ for all $i\in\setn{2,\ldots,n}$. There exists $\eps>0$ such that $\mc{\bm{p}_i}{\eps}\subset \mc{\bar{\bm{x}}}{\bar{\lambda}}$ for all $i \in\setn{2,\ldots,n}$. Therefore, $\mc{\bar{\bm{x}}+\eps (\bm{p}_1-\bar{\bm{x}})}{\bar{\lambda}}$ is another disk containing $P$ and having radius $\bar{\lambda}$. Thus we have a contradiction to Proposition~\ref{lem:uniqueness}. It follows that $\card(\ms{\bar{\bm{x}}}{\bar{\lambda}}\cap P)\geq 2$. 

Suppose now that there is an arc $A$ (that is, a connected subset of a circle) containing a semicircle of $\ms{\bar{\bm{x}}}{\bar{\lambda}}$ without points from $P$ and having endpoints $\bm{p}_1, \bm{p}_2 \in P$. Move the center $\bar{x}$ along the bisector $\bisec(\bm{p}_1,\bm{p}_2)$ towards $\frac{1}{2}(\bm{p}_1+\bm{p}_2)$ and keep $\mnorm{\bar{\bm{x}}-\bm{p}_1}$ as the radius until the center reaches $\frac{1}{2}(\bm{p}_1+\bm{p}_2)$ or a third point from $P$ hits the boundary. In the language of Lemma~\ref{lem:parametrization}, the center $\bar{x}$ is an arbitrary bisector point $\bar{\bm{x}}=\gamma(t)$. As it moves towards the midpoint $\frac{1}{2}(\bm{p}_1+\bm{p}_2)$, the parameter $t$ decreases. Thus the distance $\mnorm{\bm{p}_1-\gamma(t)}=\mnorm{\bm{p}_2-\gamma(t)}$, which coincides with the radius, also decreases.
\end{proof}
An illustration of the main steps of the proof of Theorem~\ref{thm:rademacher} can be found in Figure~\ref{fig:rademacher}.
\begin{figure}[h!]
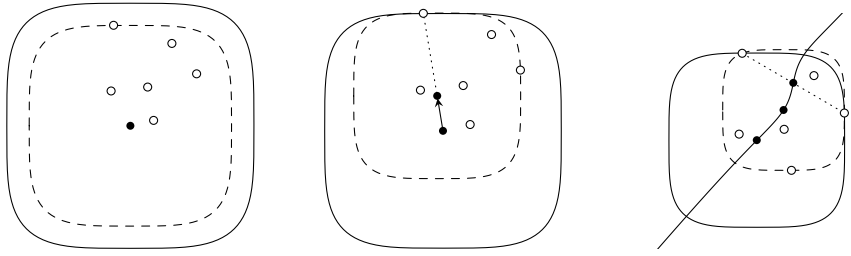

\begin{center}
\subfigure{

}
\end{center}
\caption{An $\ell^4$-norm example for the Rademacher--Toeplitz theorem. The unfilled dots are the points of $P$, the filled ones are centers or auxiliary points.}\label{fig:rademacher}
\end{figure}

From Theorem~\ref{thm:rademacher}, it follows that the minimal enclosing disk $\mc{\bar{\bm{x}}}{\bar{\lambda}}$ of a finite set $P$ is a \emph{two-point disk}, \dah there are $\bm{p}, \bm{p}^\prime\in P$ such that $\bar{\bm{x}}=\frac{1}{2}(\bm{p}+\bm{p}^\prime)$ and $\bar{\lambda}=\frac{1}{2}\mnorm{\bm{p}-\bm{p}^\prime}$, or it is the circumdisk of at least three points from $P$.

Alonso, Martini, and Spirova \cite{AlonsoMaSp2012a,AlonsoMaSp2012b} extend the notions of acuteness, rightness, and obtuseness of triangles in the following way to normed planes $(\RR^2,\mnorm{\cdot})$.
\begin{Def}
A triangle with vertices $\bm{p}_1, \bm{p}_2, \bm{p}_3 \in \RR^2$ is called \emph{norm-acute at $\bm{p}_k$} if
\begin{equation}\mnorm{\bm{p}_k -\frac{\bm{p}_i+\bm{p}_j}{2}}>\frac{\mnorm{\bm{p}_i-\bm{p}_j}}{2},\label{eq:norm_acuteness}\end{equation}
where $\setn{i,j,k}=\setn{1,2,3}$. It is called \emph{norm-right at $\bm{p}_k$} if the inequality in \eqref{eq:norm_acuteness} is changed into \enquote{$=$}, and it is called \emph{norm-obtuse at $\bm{p}_k$} if this inequality is changed into \enquote{$<$}.
\end{Def}
Figure~\ref{fig:triangle_classes} shows an example for this classification for the $\ell^4$-norm.
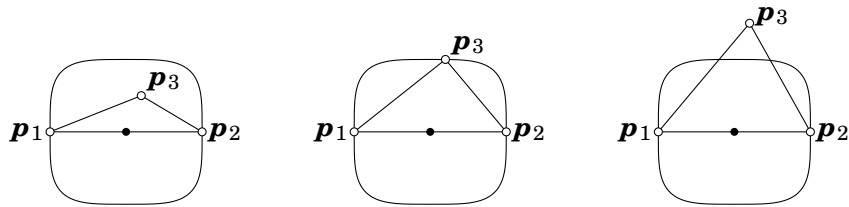
\begin{figure}[h!]\begin{center}
\begin{tikzpicture}[line cap=round,line join=round,x=1.0cm,y=0.96cm]
\draw [shift={(0,0)},scale=1] plot[domain=0:pi/2,variable=\t]({-sqrt(cos(\t r))},{sqrt(sin(\t r))})--
plot[domain=0:pi/2,variable=\t]({sqrt(sin(\t r))},{sqrt(cos(\t r))})--
plot[domain=0:pi/2,variable=\t]({sqrt(cos(\t r))},{-sqrt(sin(\t r))})--
plot[domain=0:pi/2,variable=\t]({-sqrt(sin(\t r))},{-sqrt(cos(\t r))})--cycle;
\draw [shift={(4,0)},scale=1] plot[domain=0:pi/2,variable=\t]({-sqrt(cos(\t r))},{sqrt(sin(\t r))})--
plot[domain=0:pi/2,variable=\t]({sqrt(sin(\t r))},{sqrt(cos(\t r))})--
plot[domain=0:pi/2,variable=\t]({sqrt(cos(\t r))},{-sqrt(sin(\t r))})--
plot[domain=0:pi/2,variable=\t]({-sqrt(sin(\t r))},{-sqrt(cos(\t r))})--cycle;
\draw [shift={(8,0)},scale=1] plot[domain=0:pi/2,variable=\t]({-sqrt(cos(\t r))},{sqrt(sin(\t r))})--
plot[domain=0:pi/2,variable=\t]({sqrt(sin(\t r))},{sqrt(cos(\t r))})--
plot[domain=0:pi/2,variable=\t]({sqrt(cos(\t r))},{-sqrt(sin(\t r))})--
plot[domain=0:pi/2,variable=\t]({-sqrt(sin(\t r))},{-sqrt(cos(\t r))})--cycle;
\draw (-1,0)--(1,0)--(0.2,0.5)--cycle;
\fill[color=white, draw=black] (-1,0) circle (1.5pt);
\fill[color=white, draw=black] (1,0) circle (1.5pt);
\fill[color=white, draw=black] (0.2,0.5) circle (1.5pt);
\draw (0.5,0.7) node{$\bm{p}_3$};
\draw (-1.3,0) node{$\bm{p}_1$};
\draw (1.3,0) node{$\bm{p}_2$};
\fill[color=black] (0,0) circle (1.5pt);
\draw (3,0)--(5,0)--(4.2,1)--cycle;
\fill[color=white, draw=black] (3,0) circle (1.5pt);
\fill[color=white, draw=black] (5,0) circle (1.5pt);
\fill[color=white, draw=black] (4.2,1) circle (1.5pt);
\draw (2.7,0) node{$\bm{p}_1$};
\draw (5.3,0) node{$\bm{p}_2$};
\draw (4.5,1.2) node{$\bm{p}_3$};
\fill[color=black] (4,0) circle (1.5pt);
\draw (7,0)--(9,0)--(8.2,1.5)--cycle;
\fill[color=white, draw=black] (7,0) circle (1.5pt);
\fill[color=white, draw=black] (9,0) circle (1.5pt);
\fill[color=white, draw=black] (8.2,1.5)  circle (1.5pt);
\draw (6.7,0) node{$\bm{p}_1$};
\draw (9.3,0) node{$\bm{p}_2$};
\draw (8.5,1.6) node{$\bm{p}_3$};
\fill[color=black] (8,0) circle (1.5pt);
\end{tikzpicture}\end{center}
\caption{Norm-obtuseness, norm-rightness, and norm-acuteness at $\bm{p}_3$, resp., for the $\ell^4$-norm.}\label{fig:triangle_classes}\end{figure}
The results of \cite[Section~3]{AlonsoMaSp2012b} show that the following definition provides, for normed planes, a subdivision of the family of all triangles into the subfamilies described therein.
\begin{Def}[{see \cite[Definition~3.1]{AlonsoMaSp2012b}}]\label{def:triangle_type}
A triangle $P\subset\RR^2$ is called
\begin{enumerate}[label={(\alph*)}, align=left]
\item{\emph{norm-obtuse} if it is norm-obtuse at one vertex and norm-acute at the other two vertices;}
\item{\emph{doubly norm-right} if it is norm-right at two vertices and norm-acute at the remaining vertex;}
\item{\emph{norm-right} if it is norm-right at exactly one vertex and norm-acute at the other two vertices;}
\item{\emph{norm-acute} if it is norm-acute at all three vertices.}
\end{enumerate}
\end{Def}
From \cite[Lemma~5.1]{AlonsoMaSp2012b} it follows that doubly norm-right triangles cannot occur in strictly convex normed planes. 
\section{The algorithms}\label{chap:algorithms}
In this section, we show that the algorithms by Shamos and Hoey \cite{ShamosHo1975} and Elzinga and Hearn \cite{ElzingaHe1972a}, which were designed to solve the Euclidean minimal enclosing disk problem for \emph{finite} sets $P$, can be carried over verbatim to strictly convex normed planes.
After Chrystal's algorithm \cite{Chrystal1885}, Elzinga and Hearn's algorithm was the second milestone in tackling the minimal enclosing disk problem for the Euclidean plane. Drezner and Shelah \cite{DreznerSh1987} prove its $\Omega(n^2)$ running time, where $n$ is the number of given points.
The original Elzinga--Hearn algorithm for the Euclidean plane uses acuteness, rightness, and obtuseness of triangles to reduce the problem for $n$ points to a problem for at most three points. In the Euclidean plane, there are explicit constructions for solving these  \enquote{small} problems, namely the construction of midpoints of segments and of circumcenters of acute triangles. In strictly convex normed planes $(\RR^2,\mnorm{\cdot})$, we replace the determination of triangle types via angle measures by the Thales-like method of Definition~\ref{def:triangle_type} which relies only on distance measurement.
Furthermore, let us assume that we are given a construction for finding the circumcenter of an norm-acute triangle. Then the solution of the minimal enclosing disk problem for $n$ given points can be done in exactly the same way as in the Euclidean plane.

\begin{Alg}[Elzinga/Hearn 1972]\label{alg:elzinga_plane}
\begin{algorithmic}[1]
\Require $P\subset\RR^2$, $\card(P)\geq 2$
\State Choose $\bm{p}_1,\bm{p}_2\in P$, $\bm{p}_1\neq \bm{p}_2$
\State \label{elz_step_a}$\bar{\bm{x}}\longleftarrow \frac{1}{2}(\bm{p}_1+\bm{p}_2)$, $\bar{\lambda} \longleftarrow \frac{1}{2}\mnorm{\bm{p}_1-\bm{p}_2}$
\If{$\mnorm{\bar{\bm{x}}-\bm{p}}\leq \bar{\lambda} \fall \bm{p}\in P$}
\State \Return $(\bar{\bm{x}}$, $\bar{\lambda})$
\Else \State Choose $\bm{p}_3\in P$ such that $\mnorm{\bar{\bm{x}}-\bm{p}_3}>\bar{\lambda}$
\EndIf
\If{the triangle $\setn{\bm{p}_1,\bm{p}_2,\bm{p}_3}$ is norm-obtuse or norm-right at a vertex $\bm{p}_i$, say,}\label{elz_step_b}
\State \label{elz_step_d}$\setn{\bm{p}_1,\bm{p}_2}\longleftarrow \setn{\bm{p}_1,\bm{p}_2,\bm{p}_3}\setminus \setn{\bm{p}_i}$
\State Go to~\ref{elz_step_a}
\Else\Comment{the triangle $\setn{\bm{p}_1,\bm{p}_2,\bm{p}_3}$ is norm-acute}
\State $\setn{\bar{\bm{x}}}\longleftarrow \bisec(\bm{p}_1,\bm{p}_2)\cap \bisec(\bm{p}_2,\bm{p}_3)$\label{elz_step_circumcenter}
\State $\bar{\lambda}\longleftarrow \mnorm{\bar{\bm{x}}-\bm{p}_1}$\label{elz_step_circumradius}
\EndIf
\If{$\mnorm{\bar{\bm{x}}-\bm{p}}\leq \bar{\lambda} \fall \bm{p}\in P$}
\State \Return $(\bar{\bm{x}}$, $\bar{\lambda})$
\Else \State \label{elz_step_f}Choose $\bm{p}_4\in P$ such that $\mnorm{\bar{\bm{x}}-\bm{p}_4}>\bar{\lambda}$
\State \label{elz_step_g} Choose $\bm{p}_5\in\setn{\bm{p}_1,\bm{p}_2,\bm{p}_3}$ such that $\mnorm{\bm{p}_4-\bm{p}_5}=\max_{i=1,2,3}\mnorm{\bm{p}_4-\bm{p}_i}$\;
\If{$\bm{p}_4\notin \strline{\bm{p}_5}{\bar{\bm{x}}}$}
\State \label{elz_step_c}$\bm{p}_6\longleftarrow \begin{cases}\text{the point among $\setn{\bm{p}_1,\bm{p}_2,\bm{p}_3}\setminus\setn{\bm{p}_5}$ in the half plane}\\\text{bounded by $\strline{\bm{p}_5}{\bar{\bm{x}}}$ opposite to $\bm{p}_4$}\end{cases}$
\Else \State Choose $\bm{p}_6\in \setn{\bm{p}_1,\bm{p}_2,\bm{p}_3}\setminus\setn{\bm{p}_5}$
\EndIf
\State\label{elz_step_e} $\setn{\bm{p}_1,\bm{p}_2,\bm{p}_3}\longleftarrow \setn{\bm{p}_4,\bm{p}_5,\bm{p}_6}$
\State Go to~\ref{elz_step_b}
\EndIf
\end{algorithmic}
\end{Alg}

\begin{Satz}\label{thm:elzinga}
Algorithm~\ref{alg:elzinga_plane} computes the center $\bar{\bm{x}}$ and the radius $\bar{\lambda}$ of the minimal enclosing disk of the given point set $P$.
\end{Satz}
\begin{proof}
It is easy to see that Algorithm~\ref{alg:elzinga_plane} only checks two-point disks and circumdisks determined by points of $P$ . Since there are only finitely many such disks, it suffices to show that the considered radii increase with each iteration. 
Assume there are two chosen points $\bm{p}_1,\bm{p}_2$ and we enter step~\ref{elz_step_a}. We check if the two-point disk $\mc{\bm{x}^\prime}{\lambda^\prime}$ of these two points already contains the whole set $P$. If the answer is affirmative, we are finished since no smaller disk contains $\bm{p}_1$ and $\bm{p}_2$. (This is a consequence of Proposition~\ref{lem:gritzmann}.) Otherwise there is a point outside $\mc{\bm{x}^\prime}{\lambda^\prime}$. We call it $\bm{p}_3$ and enter step~\ref{elz_step_b} with $\bm{p}_1$, $\bm{p}_2$, and $\bm{p}_3$.\\[0.5\baselineskip]
Case 1: The triangle $\setn{\bm{p}_1,\bm{p}_2,\bm{p}_3}$ is norm-obtuse at $\bm{p}_1$, say. The next disk $\mc{\bm{x}^{\prime\prime}}{\lambda^{\prime\prime}}$ under consideration is the two-point disk of $\bm{p}_2$ and $\bm{p}_3$. By \cite[Table~1]{AlonsoMaSp2012b}, we have $2\lambda^{\prime\prime}=\mnorm{\bm{p}_2-\bm{p}_3}>\mnorm{\bm{p}_1-\bm{p}_2}=2\lambda^\prime$, \dah $\lambda^{\prime\prime}>\lambda^\prime$.\\[0.6\baselineskip]
Case 2: The triangle $\setn{\bm{p}_1,\bm{p}_2,\bm{p}_3}$ is norm-right at $\bm{p}_1$, say. The next disk $\mc{\bm{x}^{\prime\prime}}{\lambda^{\prime\prime}}$ under consideration is the two-point disk of $\bm{p}_2$ and $\bm{p}_3$. By \cite[Table~1]{AlonsoMaSp2012b}, we have 
\begin{equation}
2\lambda^{\prime\prime}=\mnorm{\bm{p}_2-\bm{p}_3}\geq\mnorm{\bm{p}_1-\bm{p}_2}=2\lambda^\prime \label{eq:elzinga_right_case}
\end{equation}
with equality if the triangle $\setn{\bm{p}_1,\bm{p}_2,\bm{p}_3}$ is isosceles with $\mnorm{\bm{p}_2-\bm{p}_3}=\mnorm{\bm{p}_1-\bm{p}_2}>\mnorm{\bm{p}_1-\bm{p}_3}$ or if the triangle is equilateral. By \cite[Proposition~3.3.(b)]{AlonsoMaSp2012b}, the latter case only occurs in normed planes where the disks are parallelograms. The former case is impossible in strictly convex normed planes because
\begin{align*}
\mnorm{\frac{\bm{p}_3+\bm{p}_2}{2}-\bm{p}_1}&=\frac{\mnorm{\bm{p}_2-\bm{p}_3}}{2}\notag\\
&>\mnorm{\frac{1}{2}\lr{\frac{\bm{p}_3+\bm{p}_2}{2}-\bm{p}_1}+\frac{1}{2}\frac{\bm{p}_2-\bm{p}_3}{2}}=\frac{\mnorm{\bm{p}_2-\bm{p}_1}}{2}
\end{align*}
would yield a contradiction to the assumption that $\setn{\bm{p}_1,\bm{p}_2,\bm{p}_3}$ is isosceles. It follows that the inequality in \eqref{eq:elzinga_right_case} is strict.\\[0.6\baselineskip]
Case 3: The triangle $\setn{\bm{p}_1,\bm{p}_2,\bm{p}_3}$ is norm-acute. In step~\ref{elz_step_circumcenter} and step~\ref{elz_step_circumradius}, the circumdisk $\mc{\bm{x}^{\prime\prime}}{\lambda^{\prime\prime}}$ of $\setn{\bm{p}_1,\bm{p}_2,\bm{p}_3}$ is constructed. Note that the feasibility of step~\ref{elz_step_circumcenter} follows from the fact that norm-acute triangles in strictly convex normed planes have exactly one circumcircle, see \cite[Proposition~14~8.]{MartiniSwWe2001} and \cite[Theorem~6.1]{AlonsoMaSp2012b}. By applying Proposition~\ref{lem:gritzmann} twice, we conclude that $\lambda^{\prime\prime}>\lambda^\prime$ since
\begin{equation*}
2\lambda^{\prime\prime} = \diam(\mc{\bm{x}^{\prime\prime}}{\lambda^{\prime\prime}})> \mnorm{\bm{p}_1-\bm{p}_2}=\diam(\mc{\bm{x}^\prime}{\lambda^\prime})=2\lambda^\prime.
\end{equation*}
If $\mc{\bm{x}^{\prime\prime}}{\lambda^{\prime\prime}}$ contains the whole set $P$, it is already the solution since no smaller disk contains $\setn{\bm{p}_1,\bm{p}_2,\bm{p}_3}$. (This is a consequence of \cite[Theorem~6.4]{AlonsoMaSp2012b}.) Otherwise there is an outside point $\bm{p}_4\in P$. We cannot have 
\begin{equation*}
\mnorm{\bm{p}_4-\bm{p}_1}=\mnorm{\bm{p}_4-\bm{p}_2}=\mnorm{\bm{p}_4-\bm{p}_3}
\end{equation*}
in step~\ref{elz_step_g}, since otherwise $\bm{p}_4=\bm{x}^{\prime\prime}$, which is particularly not outside $\mc{\bm{x}^{\prime\prime}}{\lambda^{\prime\prime}}$. In other words, there are at most two points among $\bm{p}_1,\bm{p}_2,\bm{p}_3$ which are candidates for $\bm{p}_5$. The choice of $\bm{p}_6$ in step~\ref{elz_step_c} is possible by \cite[Theorem~6.3]{AlonsoMaSp2012b}, which says in particular that $\bm{x}^{\prime\prime}$ lies in the interior of the convex hull of $\setn{\bm{p}_1,\bm{p}_2,\bm{p}_3}$. Consequently, the straight line through $\bm{p}_5$ and the $\bm{x}^{\prime\prime}$ does not pass through any point from $\setn{\bm{p}_1,\bm{p}_2,\bm{p}_3}\setminus\setn{\bm{p}_5}$. If $\bm{p}_4$ lies on this straight line, we will show that it is irrelevant which of the two points from $\setn{\bm{p}_1,\bm{p}_2,\bm{p}_3}\setminus\setn{\bm{p}_5}$ is chosen as $\bm{p}_6$.
\\[0.6\baselineskip]
For that reason, let $\bm{p}_5\defeq\bm{p}_1$ be a farthest point to $\bm{p}_4$ among $\setn{\bm{p}_1,\bm{p}_2,\bm{p}_3}$. Without loss of generality, $\bm{p}_6\defeq\bm{p}_2$. We allow $\bm{p}_6$ to be a farthest point to $\bm{p}_4$ among $\setn{\bm{p}_1,\bm{p}_2,\bm{p}_3}$, \dah $\mnorm{\bm{p}_4-\bm{p}_5}=\mnorm{\bm{p}_4-\bm{p}_6}$, and we allow $\bm{p}_4\in\strline{\bm{p}_5}{\bm{x}^{\prime\prime}}$. \\[0.6\baselineskip]
Case 3.1: The triangle $\setn{\bm{p}_4,\bm{p}_5,\bm{p}_6}$ is norm-obtuse at $\bm{p}_5$. This would imply that $\mnorm{\bm{p}_4-\bm{p}_5}<\mnorm{\bm{p}_4-\bm{p}_6}$, see \cite[Table~1]{AlonsoMaSp2012b}. This is a contradiction to our assumptions.\\[0.6\baselineskip]
Case 3.2: The triangle $\setn{\bm{p}_4,\bm{p}_5,\bm{p}_6}$ is norm-obtuse at $\bm{p}_6$. Then $\bm{x}^{\prime\prime\prime}=\frac{1}{2}(\bm{p}_4+\bm{p}_5)$ is the new center and $\lambda^{\prime\prime\prime}=\frac{1}{2}\mnorm{\bm{p}_4-\bm{p}_5}$ is the new radius. By norm-obtuseness at $\bm{p}_6$, we have $\mnorm{\bm{x}^{\prime\prime\prime}-\bm{p}_6}<\mnorm{\bm{x}^{\prime\prime\prime}-\bm{p}_5}$ and, consequently, $\bm{x}^{\prime\prime\prime}\in\bisec(\bm{p}_5,\bm{p}_6)$. If we assume $\lambda^{\prime\prime\prime}\leq \lambda^{\prime\prime}$, then $\bm{x}^{\prime\prime\prime}$ does not lie on $\strline{\bm{p}_5}{\bm{x}^{\prime\prime}}$ since otherwise $\bm{p}_4\in\mc{\bm{x}^{\prime\prime}}{\lambda^{\prime\prime}}$. Hence $\bm{x}^{\prime\prime\prime}$ lies in the interior of the shaded region in Figure~\ref{fig:elzinga}.
Furthermore, we have $\mnorm{\bm{x}^{\prime\prime\prime}-\bm{p}_5}=\lambda^{\prime\prime\prime}\leq \lambda^{\prime\prime}$ and thus the interior of the shaded region intersects $\mc{\bm{p}_5}{\lambda^{\prime\prime}}$. In particular, the intersection of $\bisec(\bm{p}_5,\bm{p}_6)$ and $\mc{\bm{p}_5}{\lambda^{\prime\prime}}$ contains a point $\bm{y}$. But $\bm{y}$ lies \enquote{strictly afterwards} $\bm{x}^{\prime\prime}$ on the bisector of $\bm{p}_5$ and $\bm{p}_6$ (in the sense of Lemma~\ref{lem:parametrization}), \dah $\lambda^{\prime\prime}\geq\mnorm{\bm{y}-\bm{p}_5}>\mnorm{\bm{x}^{\prime\prime}-\bm{p}_5}=\lambda^{\prime\prime}$. This is a contradiction.
\begin{figure}[h!]
\begin{center}

\end{center}
\caption{Proof of Theorem~\ref{thm:elzinga}: Case 3.\label{fig:elzinga}}
\end{figure}
\newpage
Case 3.3: In any other case, we have $\bm{x}^{\prime\prime\prime}\in\bisec(\bm{p}_5,\bm{p}_6)$. If we assume $\lambda^{\prime\prime}\geq\lambda^{\prime\prime\prime}$, then $\bm{x}^{\prime\prime\prime}\in\co\setns{\bm{p}_5,\bm{p}_6,\bm{x}^{\prime\prime}}$, see \cite[Proposition~18]{MartiniSwWe2001} and Lemma~\ref{lem:parametrization}. The straight line through $\bm{p}_4$ and $\bm{x}^{\prime\prime}$ separates $\co\setns{\bm{p}_5,\bm{p}_6,\bm{x}^{\prime\prime}}$ into two parts, namely $\co\setns{\bm{s},\bm{p}_6,\bm{x}^{\prime\prime}}$ and $\co\setns{\bm{p}_5,\bm{s},\bm{x}^{\prime\prime}}$ as depicted in Figure~\ref{fig:case33}.
Note that although it is the case in Figure~\ref{fig:case33}, it is not clear whether the part of $\bisec(\bm{p}_5,\bm{p}_6)$ between $\bm{x}^{\prime\prime}$ and $\frac{1}{2}(\bm{p}_5+\bm{p}_6)$ is fully contained in one of the sets $\co\setns{\bm{s},\bm{p}_6,\bm{x}^{\prime\prime}}$ and $\co\setns{\bm{p}_5,\bm{s},\bm{x}^{\prime\prime}}$.\\[0.6\baselineskip]
Case 3.3.1: If $\bm{x}^{\prime\prime\prime}\in\co\setns{\bm{p}_5,\bm{s},\bm{x}^{\prime\prime}}$, then $\co\setns{\bm{p}_4,\bm{x}^{\prime\prime},\bm{p}_6}\subset\co\setns{\bm{p}_4,\bm{x}^{\prime\prime\prime},\bm{p}_6}$. Now \cite[Corollary~28]{MartiniSwWe2001} yields
\begin{align*}
2\lambda^{\prime\prime\prime}&\geq\mnorm{\bm{p}_4-\bm{x}^{\prime\prime\prime}}+\mnorm{\bm{p}_6-\bm{x}^{\prime\prime\prime}}\\
&>\mnorm{\bm{p}_4-\bm{x}^{\prime\prime}}+\mnorm{\bm{p}_6-\bm{x}^{\prime\prime}}\\
&>2\lambda^{\prime\prime},
\end{align*}
a contradiction to the assumption.\\[0.6\baselineskip]
Case 3.3.2: If $\bm{x}^{\prime\prime\prime}\in\co\setns{\bm{s},\bm{p}_6,\bm{x}^{\prime\prime}}$, then $\bm{p}_4$, $\bm{x}^{\prime\prime}$, $\bm{x}^{\prime\prime\prime}$ and $\bm{p}_6$ form in this order a convex quadrangle. Now \cite[Proposition~7]{MartiniSwWe2001} yields
\begin{align*}
&&\mnorm{\bm{x}^{\prime\prime}-\bm{p}_4}+\mnorm{\bm{x}^{\prime\prime\prime}-\bm{p}_6}&<\mnorm{\bm{x}^{\prime\prime\prime}-\bm{p}_4}+\mnorm{\bm{x}^{\prime\prime}-\bm{p}_6}\\
&\Longrightarrow\quad&\mnorm{\bm{x}^{\prime\prime}-\bm{p}_4}&<\underbrace{\mnorm{\bm{x}^{\prime\prime\prime}-\bm{p}_4}-\mnorm{\bm{x}^{\prime\prime\prime}-\bm{p}_6}}_{\leq 0}+\mnorm{\bm{x}^{\prime\prime}-\bm{p}_6}\\
&\Longrightarrow& \lambda^{\prime\prime}<\mnorm{\bm{x}^{\prime\prime}-\bm{p}_4}&<\mnorm{\bm{x}^{\prime\prime}-\bm{p}_6}=\lambda^{\prime\prime},
\end{align*}
a contradiction.
\end{proof}
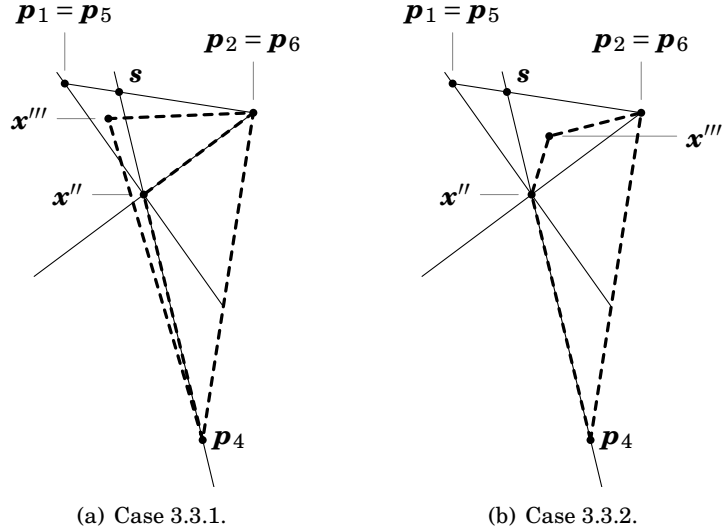
\begin{figure}[h!]
\begin{center}
\subfigure[Case 3.3.1.]{
\begin{tikzpicture}[line cap=round,line join=round,>=triangle 45,x=1.55cm,y=1.55cm]
\clip (-1.3,-2.5) rectangle (1.5,1.8);
\fill [color=black] (0.93,0.7) circle (1.5pt) node[circle,pin=90:{$\bm{p}_2=\bm{p}_6$}] {};
\fill [color=black] (-0.67,0.95) circle (1.5pt) node[circle,pin=90:{$\bm{p}_1=\bm{p}_5$}] {};
\fill [color=black] (0.5,-2.1) circle (1.5pt) node[right] {$\bm{p}_4$};
\fill [color=black] (0,0) circle (1.5pt) node[circle,pin=180:{$\bm{x}^{\prime\prime}$}] {};
\fill [color=black] (-0.3,0.65) circle (1.5pt) node[circle,pin=180:{$\bm{x}^{\prime\prime\prime}$}] {};
\draw ($(-0.67,0.95)!1.0!(0.67,-0.95)$)--($(0.67,-0.95)!1.05!(-0.67,0.95)$);
\draw[name path=strline1] (-0.67,0.95)--(0.93,0.7);
\draw (0.93,0.7)--($(0.93,0.7)!2.0!(0,0)$);
\draw[name path=strline2] ($(0.5,-2.1)!1.5!(0,0)$)--($(0,0)!1.3!(0.5,-2.1)$);
\fill[color=black,name intersections={of=strline1 and strline2}](intersection-1)circle(1.5pt)node[above right]{$\bm{s}$};
\draw[line width=1.2pt,dashed] (0.5,-2.1)--(0,0)--(0.93,0.7)--(-0.3,0.65)--(0.5,-2.1)--(0.93,0.7);
\end{tikzpicture}}
\qquad
\subfigure[Case 3.3.2.]{
\begin{tikzpicture}[line cap=round,line join=round,>=triangle 45,x=1.55cm,y=1.55cm]
\clip (-1.2,-2.5) rectangle (1.8,1.8);
\fill [color=black] (0.93,0.7) circle (1.5pt) node[circle,pin=90:{$\bm{p}_2=\bm{p}_6$}] {};
\fill [color=black] (-0.67,0.95) circle (1.5pt) node[circle,pin=90:{$\bm{p}_1=\bm{p}_5$}] {};
\fill [color=black] (0.5,-2.1) circle (1.5pt) node[right] {$\bm{p}_4$};
\fill [color=black] (0,0) circle (1.5pt) node[circle,pin=180:{$\bm{x}^{\prime\prime}$}] {};
\fill [color=black] (0.15,0.5) circle (1.5pt) node[circle,pin={[pin distance=1.5cm]0:{$\bm{x}^{\prime\prime\prime}$}}] {};
\draw ($(-0.67,0.95)!1.0!(0.67,-0.95)$)--($(0.67,-0.95)!1.05!(-0.67,0.95)$);
\draw[name path=strline1] (-0.67,0.95)--(0.93,0.7);
\draw (0.93,0.7)--($(0.93,0.7)!2.0!(0,0)$);
\draw[name path=strline2] ($(0.5,-2.1)!1.5!(0,0)$)--($(0,0)!1.3!(0.5,-2.1)$);
\fill[color=black,name intersections={of=strline1 and strline2}](intersection-1)circle(1.5pt)node[above right]{$\bm{s}$};
\draw[line width=1.2pt,dashed] (0.5,-2.1)--(0,0)--(0.15,0.5)--(0.93,0.7)--cycle;
\end{tikzpicture}}
\end{center}
\caption{Proof of Theorem~\ref{thm:elzinga}: Case 3.3.}\label{fig:case33}
\end{figure}
\begin{Bem}
In fact, the interior of the shaded region in Figure~\ref{fig:elzinga} has only one connected component. This follows from a sharpening of \cite[Lemma~18]{MartiniSwWe2001}, which reads as follows. Suppose the unit circle of a normed plane $(\RR^2,\mnorm{\cdot})$ does not contain a line segment parallel to the straight line through $\bm{p}$ and $\bm{q}$. Then, for every point $\bm{z}\in\bisec(\bm{p},\bm{q})$, the following relation holds:
\begin{equation*}
\bisec(\bm{p},\bm{q})\setminus \setn{\bm{z}} \subset \setcond{\bm{z}+\lambda (\bm{z}-\bm{p})+\mu(\bm{z}-\bm{q})}{\lambda\mu >0}.
\end{equation*}
For the proof of this statement it suffices to assume the existence of a point $\bm{w} \in \bisec(\bm{p},\bm{q})$ which is distinct from $\bm{z}$ and lies, say, in $\setcond{\alpha \bm{z}+(1-\alpha)\bm{p}}{\alpha \geq 1}$. Since $\bm{z},\bm{w}\in\bisec(\bm{p},\bm{q})$, we have $\mnorm{\bm{z}-\bm{p}}=\mnorm{\bm{z}-\bm{q}}$ and $\mnorm{\bm{w}-\bm{p}}=\mnorm{\bm{w}-\bm{q}}$. The collinearity of the points $\bm{p}$, $\bm{z}$, and $\bm{w}$ yields $\mnorm{\bm{w}-\bm{p}}=\mnorm{\bm{w}-\bm{z}}+\mnorm{\bm{z}-\bm{p}}$. Substituting the term on the left-hand side and the second one on the right-hand side, we obtain $\mnorm{\bm{w}-\bm{q}}=\mnorm{\bm{w}-\bm{z}}+\mnorm{\bm{z}-\bm{q}}$. By \cite[Lemma~1]{MartiniSwWe2001},
we conclude that the unit circle contains a line segment parallel to the straight line through $\bm{p}$ and $\bm{q}$, a contradiction.
\end{Bem}
\begin{Bem}
As already mentioned, Algorithm~\ref{alg:elzinga_plane} requires an additional subroutine which computes the circumdisk of a given triangle. Assuming that this computation can be done in constant time, the running time of Algorithm~\ref{alg:elzinga_plane} will still show an $\Omega(n^2)$ behaviour like in the Euclidean case.
\end{Bem}
In the early years of Computational Geometry, Shamos and Hoey \cite{ShamosHo1975} proposed an algorithm for the minimal enclosing disk problem for $n$ given points in the Euclidean plane. Their algorithm is based on the construction of farthest-point Voronoi diagrams. For that purpose, they use a divide-and-conquer technique to obtain $O(n\log n)$ running time.
For wider classes of norms, constructions of Voronoi diagrams and respective running time results are included in the papers of Lee \cite{Lee1980, Lee1982}, Chew and Drysdale \cite{ChewDr1985}, and Chazelle and Edelsbrunner \cite{ChazelleEd1987}. The simple structure of farthest-point Voronoi diagrams enables an $O(n)$ search for the optimal disk once the diagram is constructed.
\begin{Alg}[Shamos/Hoey 1975]\label{alg:shamos_plane}
\begin{algorithmic}[1]
\Require $n\geq 2$, $P=\setn{\bm{p}_1,\ldots,\bm{p}_n}\subset\RR^2$
\State Construct the farthest-point Voronoi diagram with respect to $\bm{p}_1,\ldots,\bm{p}_n$
\ForEach{edge of the diagram}\label{shamos_step_a}
\State Determine the distance between the two defining points
\EndFor
\State \label{shamos_step_e} Find the maximum among these distances
\If{the two-point disk of the corresponding points contains $\bm{p}_1,\ldots,\bm{p}_n$}
\State \label{shamos_step_b}\Return the center and the radius of this disk
\Else \ForEach{vertex of the diagram}\label{shamos_step_c}
\State Compute its distance to one of its defining points
\EndFor
\State \label{shamos_step_f} Find the minimum among these distances
\State \label{shamos_step_d}\Return the corresponding vertex of the Voronoi diagram and the minimum distance
\EndIf
\end{algorithmic}
\end{Alg}

\begin{Satz}
Algorithm~\ref{alg:shamos_plane} computes the center and the radius of the minimal enclosing disk of the given point set $P$.
\end{Satz}

\begin{proof}
Let $\mc{\bar{\bm{x}}}{\bar{\lambda}}$ be the minimal enclosing disk of $P$. By Theorem~\ref{thm:rademacher}, $\ms{\bar{\bm{x}}}{\bar{\lambda}}\cap P$ contains at least two points. If it contains exactly two points $\bm{p}_1$, $\bm{p}_2$, the center $\bar{\bm{x}}$ belongs to the farthest-point Voronoi regions of $\bm{p}_1$ and $\bm{p}_2$ but not to any other farthest-point Voronoi region, \dah $\bar{\bm{x}}$ lies on the edge of the diagram that belongs to $\bm{p}_1$ and $\bm{p}_2$. Taking Theorem~\ref{thm:rademacher} into account, it follows that $\bar{\bm{x}}=\frac{1}{2}(\bm{p}_1+\bm{p}_2)$. Hence
\begin{equation}
\mnorm{\bm{p}_1-\bm{p}_2}=\diam(\mc{\bar{\bm{x}}}{\bar{\lambda}})>\mnorms{\bm{p}-\bm{p}^\prime}\text{ for all }\bm{p},\bm{p}^\prime\in P\setminus\setn{\bm{p}_1,\bm{p}_2}.\label{eq:max_edge}
\end{equation} If $\ms{\bar{\bm{x}}}{\bar{\lambda}}\cap P$ contains at least three points, then $\bar{\bm{x}}$ lies in the intersection of at least three farthest-point Voronoi regions, \dah $\bar{\bm{x}}$ is a vertex of the diagram. In step~\ref{shamos_step_e}, we are looking for the maximum distance of pairs of points which determine edges of the diagram. Then, by~\eqref{eq:max_edge}, the two-point disk of the corresponding points realizing this maximum is the minimal enclosing disk of $P$ if it contains $P$.
If this is not the case, the center of the minimal enclosing disk has to be a vertex of the diagram. Clearly, each disk, which is centered at a vertex of the diagram and contains the (at least three) points that determine the farthest-point Voronoi regions to which the vertex belongs, contains $P$. Thus it suffices to find the smallest disk belonging, in the sense just explained, to a vertex, see step~\ref{shamos_step_f}.
\end{proof}

\section{Conclusion}
In the present paper, two algorithms for solving the minimal enclosing disk problem are investigated. Here the planar Euclidean setting is being replaced by strictly convex norms on $\RR^2$. Further research in this direction might include the generalization to arbitrary norms or even to gauges. (These are distance functions whose unit balls are still convex compact sets having the origin as interior point but no longer have to be centered at the origin.) For non-strictly convex norms, bisectors might have interior points which requires a careful definition of Voronoi cells. Moreover, there is a fourth triangle type (doubly right triangles) in those planes: It has to be checked whether the incorporation of this triangle class into a Elzinga--Hearn-type algorithm is possible. For gauges, the shape of bisectors and Voronoi diagrams is known \cite{Ma2000}, but the analogous theory for \cite{AlonsoMaSp2012a,AlonsoMaSp2012b} is completely missing.

\providecommand{\bysame}{\leavevmode\hbox to3em{\hrulefill}\thinspace}

\end{document}